\documentclass[12pt,reqno,a4paper]{amsart}%
\usepackage{amsfonts,amsmath,amssymb,amsthm,hyperref}

\usepackage[sc]{mathpazo} 
\usepackage[scaled]{helvet} 
\usepackage{eulervm} 

\ifx\pdfoutput\undefined
\usepackage[dvips]{graphicx}
\else
\usepackage[pdftex]{graphicx}
\allowdisplaybreaks

\newtheorem{theorem}{Theorem}
\newtheorem{lemma}{Lemma}
\newtheorem{prop}{Proposition}
\newtheorem{coroll}{Corollary}
\theoremstyle{definition}
\newtheorem{remark}{Remark}
\newtheorem{example}{Example}
\newtheorem{bijection}{Bijection}
\newtheorem{definition}{Definition}
\newtheorem{notation}{Notation}

\DeclareMathOperator{\chara}{char}
\DeclareMathOperator{\val}{val}
\DeclareMathOperator{\pos}{pos}

\DeclareMathOperator{\grad}{deg}
\newcommand{\lAngle}{\ensuremath{\langle\hspace{-0.1cm}\langle}}
\newcommand{\rAngle}{\ensuremath{\rangle\hspace{-0.1cm}\rangle}}
\newcommand{\lbrakk}{\ensuremath{[\hspace{-0.15cm}[}}
\newcommand{\rbrakk}{\ensuremath{]\hspace{-0.15cm}]}}

\newcommand{\Seq}{\textsc{SEQ}}

\newcommand{\N}{\ensuremath{\mathbb{N}}}

\newcommand\cB{\mathcal B}

\newcommand{\refF}[1]{Figure~\ref{#1}}

\newcommand\JCTA{\emph{J. Combin. Theory Ser. A} }

\def\nobibitem#1\par{}

\begin{document}

\title{On Path diagrams and Stirling permutations}
\date{}

\author[M.~Kuba]{Markus Kuba}
\address{Markus Kuba\\
Department of Applied Mathematics and Physics\\
University of Applied Sciences - Technikum Wien\\
H\"ochst\"adtplatz 5, 1200 Wien} %
\email{kuba@technikum-wien.at}

\author[A.~L.~Varvak]{Anna L.~Varvak}
\address{Anna L.~Varvak\\
Soka University of America\\
Pauling Hall 428\\
1 University Drive\\
Aliso Viejo, CA 92656, USA} %
\email{avarvak@soka.edu}

\maketitle

\begin{abstract}
A permutation can be locally classified according to the four local types: peaks, valleys, double rises and double falls. 
The corresponding classification of binary increasing trees uses four different types of nodes. 
Flajolet demonstrated the continued fraction representation of the generating function of local types, using a classical bijection between permutations, binary increasing trees, and suitably defined path diagrams induced by Motzkin paths.

The aim of this article is to extend the notion of local types from permutations to $k$-Stirling permutations (also known as $k$-multipermutations).
We establish a bijection of these local types to node types of $(k+1)$-ary increasing trees. We present a branched continued fraction representation of the generating function of these local types through a bijection with path diagrams induced by \L ukasiewicz paths, generalizing the results from permutations to arbitrary $k$-Stirling permutations.

We further show that the generating function of ordinary Stirling permutation has at least three branched continued fraction representations, using correspondences between non-standard increasing trees, $k$-Stirling permutations and path diagrams.
\end{abstract}


\section{Introduction\label{STIRPsec1}}
Any ordinary permutation $\sigma=\sigma_1\dots\sigma_n$ of size $n$ can be locally be classified according to four local types called peaks (maxima), valleys (minima), double rises and double falls, depending on the relative order of $\sigma_j$ to its neighbor; see Flajolet~\cite{Flajo1980}, or Conrad and Flajolet~\cite{Conrad2006} and the references therein. By the classical bijection between ordinary permutations and binary increasing trees~\cite{Flajo1980,BerFlaSal1992}, these local types correspond with node types of binary increasing trees~\cite{Flajo1980,Conrad2006}. Through a bijection between permutations and path diagrams induced by Motzkin paths, Flajolet~\cite{Flajo1980} obtained a continued fraction representation of the ordinary generating function of permutations $P(z;u,v,w)$ with respect to the local types (see also Fran\c{c}on and Viennot~\cite{Francon1979}), leading to interesting continued fraction representations concerning Euler numbers, tangent numbers, etc., can be obtained by specific evaluations of the variables $u$, $v$ and $w$.

The main aim of this article is to extend the notion of local types to a generalized version of permutations called the $k$-Stirling permutations, establish a relation of these local types with node types of $(k+1)$-ary increasing trees, and to obtain a branched continued fraction representation of the generating function of these local types through a bijection with suitably defined path diagrams induced by \L ukasiewicz paths, extending the following theorem: 

\begin{theorem}[Flajolet - Continued Fraction for Permutations]
Let $P_{k,\ell,m}$ be the number of permutations having $k$ minima (hence
$k + 1$ maxima), $\ell$ double rises and $m$ double falls. The ordinary generating function
$P(z;u,v,w)=\sum P_{k,\ell,m}u^k v^{\ell}  w^{m} z^{2k+1+\ell+m+1}$ 
has the expression 
\[
P(z;u,v,w)=\frac{1}{1-1(v+w)z-\frac{1\cdot 2uz^2}{1-2(v+w)z-\frac{2\cdot 3 uz^2}{\dots}}}.
\]
\end{theorem}

We note that related very general results on branched continued fraction expansions have been recently given by P\'etr\'eolle, Sokal and Zhu~\cite{PSZ2018}.

\smallskip

The family of Stirling permutations $\mathcal{Q}$ were introduced by Gessel and Stanley
\cite{GessStan1978} in relation to the Stirling numbers of first and second kind and their connection to properties of Eulerian numbers. A Stirling permutation $\sigma=\sigma_1\dots\sigma_{2n}\in\mathcal{Q}_n$ is a permutation of
the multiset $\{1, 1, 2, 2, \dots , n, n\}$ such that for each
$i$, $1\le i \le n$, the elements between the two
occurrences of $i$ are greater than or equal to $i$.  Recently, this class of combinatorial objects have generated some interest:
B\'ona~\cite{Bona2007} studied the distribution of
descents in Stirling permutation, and Janson~\cite{Jan2008} showed the connection between Stirling permutations and plane-oriented recursive trees, and proved a joint normal limit law for the parameters considered by B\'ona.

\smallskip

A natural generalization of Stirling permutations is to consider the family $\mathcal{Q}(k)$ of
permutations of a more general multiset
$\{1^{k}, 2^{k}, \dots, n^{k}\}$, with $k\in \{1,2,\ldots\}$, with the restriction that
for each $i$, $1\le i \le n$, the elements between two consecutive occurrences of
$i$ are greater than $i$, which we call $k$-Stirling permutations. Such generalized Stirling permutations have already previously been considered by Brenti~\cite{Brenti1989, Brenti1998}, and also by Park~\cite{Park1994b,Park1994a,Park1994c} under the name of $k$-multipermutations. The case $k=1$ corresponds to ordinary permutations $\mathcal{Q}_n(1)=S_n$, and the case $k=2$ corresponds to Stirling permutations $\mathcal{Q}=\mathcal{Q}(2)$.
Recently, Janson et al.~\cite{JanKuPa2008} studied several parameters in $k$-Stirling permutations, related to the studies~\cite{Bona2007,Jan2008}, extending the results of~\cite{Bona2007,Jan2008} concerning the distribution of descents and related statistics.
An important result of~\cite{JanKuPa2008} is the natural bijection between $k$-Stirling permutations and $(k+1)$-ary increasing trees, which was already known to Gessel (see Park~\cite{Park1994b}).

\smallskip

This article is organized as follows.
In section~\ref{SEC IncTrStirPerm}, we define $k$-Stirling permutations and increasing trees, and give a bijection between $k$-Stirling permutations and $(k+1)$-ary increasing trees.

In section~\ref{SEC localtypes}, we extend the notion of local types to $k$-Stirling permutations, and establish a bijection between these local types and node types of $(k+1)$-ary increasing trees.

In section~\ref{SEC PathLocTypes}, we define path diagrams induced by \L ukasiewicz paths, and establish a bijection between the path diagrams with appropriate possibility function and restrictions on the plane paths and $k$-Stirling permutations.  We then give a branched continued fraction representation of the ordinary generating function of the $k$-Stirling permutations with respect to the local types.

Finally, in section~\ref{SEC Stirling}, we first go back to the classical Stirling permutations, presenting an alternative branched continued fraction representation using unrestricted \L ukasiewicz paths and the correspondence between Stirling permutations and plane-oriented recursive trees~\cite{Jan2008}.  We further provide statistics in Stirling permutations and in ternary increasing trees which are combinatorially equivalent to the statistics of the number of nodes of outdegree $j$ in a plane-oriented recursive tree of size $n$.
Then, we present two general path diagram representations for $k$-Stirling permutations using path diagrams with unrestricted \L ukasiewicz paths and families of non-standard increasing trees, allowing to construct two additional branched continued fraction expansions for the generating function of Stirling permutations.

\subsection{Acknowledgments}
The authors thank Ira Gessel for connecting them in the beginning of this work. A great many thanks to Alan P. Sokal for his interest in this work and his encouraging remarks, sparking new life into this project. 

\smallskip

The authors thank the referee for the constructive and very helpful remarks, improving the presentation of this work.

\smallskip

During the early stages of this work the first author was supported by the Austrian Science Foundation FWF, grant S9608.

\section{Increasing trees and generalized Stirling permutations}\label{SEC IncTrStirPerm}

\subsection{Generalized Stirling permutations}

\begin{definition}\label{k-stir}
Let $\{1^k, 2^k, \ldots, n^k\}$ denote a multiset where each element in $\{1,2,\ldots,n\}$ occurs $k$ times.
A \emph{$k$-Stirling permutation  of size $n$}, denoted as $\tau\in\mathcal{Q}_n(k)$,  is a permutation $\tau_1 \tau_2 \ldots \tau_{kn}$ of elements in the multiset $\{1^k, 2^k, \ldots, n^k\}$  with the condition that for $i<j<k$, if $\tau_i=\tau_k$, then $\tau_i\le \tau_j \ge \tau_k$.
\end{definition}
For example, $224442113331$ is a $3$-Stirling permutation of size $3$.

Let $Q_n(k)=|\mathcal{Q}_n(k)|$ denote the number of $k$-Stirling permutations of size $n$. Then

\begin{equation}\label{qn}
Q_n(k)
= \prod_{i=0}^{n-1}(ki+1)
=k^{n-1}\frac{\Gamma(n+1/k)}{\Gamma(1/k)},
\end{equation}

\noindent
by induction on  $n$: a $k$-Stirling permutation of size $n$ can be obtained from a $k$-Stirling permutation of size $(n-1)$ by inserting
the $k$ copies of $n$ as a substring into any of the $k(n-1)+1$ positions between the existing elements, including the first and the last position; see for example~\cite{Park1994b,JanKuPa2008}.

For example, in the case $k=3$, there is only one $3$-Stirling permutation of size $1$, given by $111$ ; there are four $3$-Stirling permutations of size $2$, given by $111222$, $112221$, $122211$, $222111$; etc.  And in particular, the number of $2$-Stirling permutations is $Q_n(2)=(2n-1)!!$.

The correspondence between ordinary permutations and binary trees has been useful in uncovering the internal structure of permutations, and giving alternative combinatorial models to quantities that relate to the local types of permutations, like the Eulerian numbers that count the number of permutations with a certain number of descents.  Likewise, the correspondence between $2$-Stirling permutations and plane-oriented recursive trees proved useful in studying the distribution of descents in Stirling permutations~\cite{Jan2008}.  In the next section, we establish the correspondence for the general case of $k$-Stirling permutations and $(k+1)$-ary increasing trees.

\subsection{Families of increasing trees}

We introduce a general family of increasing trees based on earlier considerations of
Bergeron et al.~\cite{BerFlaSal1992} and Panholzer and Prodinger~\cite{PanPro2005+}, of which $(k+1)$-ary increasing trees and plane-oriented recursive trees are special cases.
These tree families and their combinatorial description are quite well known;
we collect some relevant results of~\cite{BerFlaSal1992,PanPro2005+,JanKuPa2008} for the reader's convenience.

Informally, an increasing tree of size $n$ is a rooted ordered tree with $n$ nodes, where
the nodes are labeled by distinct integers of the set
$\{1, \dots, n\}$ such that each sequence of labels along any path starting at the root is increasing. These are the simple families of increasing trees introduced in~\cite{BerFlaSal1992}; the underlying unlabeled tree model
is the so-called simply generated tree~\cite{MeiMoo1978}.

Formally, a class $\mathcal{T}$ of a simple family of increasing trees can be defined
in the following way. A sequence of non-negative numbers $(\varphi_{\ell})_{\ell \ge 0}$
with $\varphi_{0} > 0$, called the \emph{degree-weight sequence},
is used to define the weight $w(T)$ of any ordered tree $T$ by
$w(T) := \prod_{v} \varphi_{\grad^{+}(v)}$, where $v$ ranges over all vertices of $T$, and $\grad^{+}(v)$ is the
out-degree of the vertex $v$. Let $\mathcal{L}(T)$ denote the set of increasing labelings of the  of the ordered tree $T$ with distinct integers $\{1, 2, \dots, |T|\}$,
where $|T|$ denotes the size of the tree $T$.
Then the simple family of increasing trees $\mathcal{T}$ consists of all ordered trees $T$
together with their weights $w(T)$
and the increasing labeling $\lambda \in\mathcal{L}(T)$. 
The simple family of increasing trees $\mathcal{T}$ associated with a degree-weight generating function
$\varphi(t) := \sum_{\ell \ge 0} \varphi_{\ell} t^{\ell}$ can be described
by the formal recursive equation
\begin{equation}
   \label{eqnz0}
   \mathcal{T} = \bigcirc\hspace*{-0.75em}\text{\small{$1$}}\hspace*{0.4em}
   \times \Big(\varphi_{0} \cdot \{\epsilon\} \; \dot{\cup} \;
   \varphi_{1} \cdot \mathcal{T} \; \dot{\cup} \; \varphi_{2} \cdot
   \mathcal{T} \ast \mathcal{T} \; \dot{\cup} \; \varphi_{3} \cdot
   \mathcal{T} \ast \mathcal{T} \ast \mathcal{T} \; \dot{\cup} \; \cdots \Big)
   = \bigcirc\hspace*{-0.75em}\text{\small{$1$}}\hspace*{0.3em} \times \varphi(\mathcal{T}),
\end{equation}
where $\bigcirc\hspace*{-0.75em}\text{\small{$1$}}\hspace*{0.4em}$ denotes the node
labeled by $1$, $\times$ the cartesian product, $\dot{\cup}$ the disjoint union, $\ast$ the partition product for labeled
objects; $\varphi(\mathcal{T})$ the vertices substituted structure (e.~g., see~\cite{VitFla1990,FlaSed2008}). For a given degree-weight sequence, we define the total weight of size $n$ increasing trees
by $T_{n} := \sum_{|T|=n} w(T) \cdot L(T)$, where $L(T) :=\big|\mathcal{L}(T)\big|$ is the  number of distinct increased labelings on the ordered tree $T$. It follows from the recursive structure of increasing trees that the exponential generating function
$T(z) := \sum_{n \ge 1} T_{n} \frac{z^{n}}{n!}$
satisfies the autonomous first order differential equation
\begin{equation}
   \label{eqnz1}
   T'(z) = \varphi\big(T(z)\big), \quad T(0)=0.
\end{equation}
We obtain the families of $(k+1)$-ary increasing trees and plane-oriented recursive trees by choosing appropriate degree-weight sequences $(\varphi_{\ell})_{\ell \ge 0}$.

\smallskip

\begin{example}
The family $\mathcal{T}=\mathcal{T}(k+1)$ of $(k+1)$-ary increasing trees, with integer $k\in \N$, is the family of increasing trees
where each node has $k+1$ (labeled) positions for children, going from left to right. The vacant positions are usually denoted by external nodes (see Figure~\ref{STIRPfig0} for an illustration of ternary increasing trees).
Each node can have $0\le \ell \le k+1$ internal child nodes, with $\binom {k+1}{\ell}$ different ways to
attach them (see Figure~\ref{TYPESfig1}).  Thus the appropriate degree-weight sequence to specify the weight of the nodes is
$\varphi_{\ell}=\binom{k+1}{\ell}$ for $0\le \ell\le k+1$, $\varphi_{\ell}=0$ for $k+1<\ell$.
Consequently, the degree weight generating function is $ \varphi(t)=\sum_{\ell\ge 0}\varphi_{\ell} t^{\ell} =(1+t)^{k+1}$, which lets us derive the exponential generating function $T(z)=T(z,k+1)$  of $(k+1)$-ary increasing trees
by solving the corresponding differential equation~\eqref{eqnz1}. The number $T_n=T_n(k+1)=|\mathcal{T}_n(k+1)|$ of $(k+1)$-ary increasing trees of size $n$ can be obtained from the generating function, or by induction on  $n$:

\begin{equation}
\label{eqnkary}
T(z)=\frac{1}{(1-kz)^{\frac1k}}-1,
\qquad T_{n}=\prod_{\ell=1}^{n}(k(\ell-1)+1)=k^{n-1}\frac{\Gamma(n+1/k)}{\Gamma(1/k)},
\quad
n\ge 1.
\end{equation}
The case $k=1$ is the family of binary increasing trees, and the case $k=2$ is the family of ternary increasing trees. Note that $T_n=Q_n$, the number of $k$-Stirling permutations of size $n$ (see equation~\eqref{qn}).
\end{example}

\begin{figure}[!htb]
\centering
\includegraphics[angle=0,scale=1]{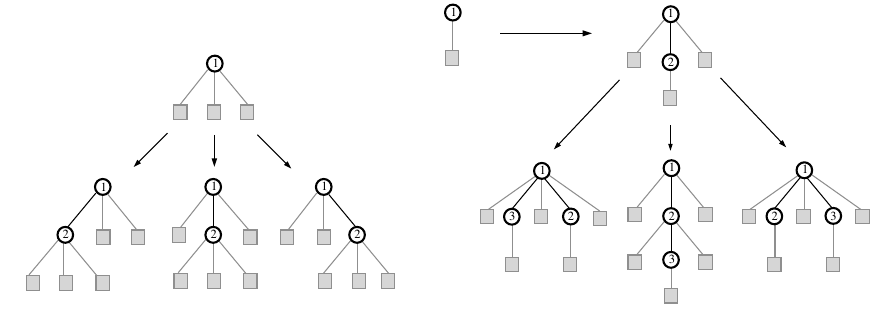}
\caption{Ternary increasing trees of size one and two and plane-oriented increasing trees of size one, two and three, respectively. The positions where new nodes can be attached are denoted by external nodes.\label{STIRPfig0}}
\end{figure}

\begin{example}
The family $\mathcal{P}$ of plane-oriented recursive trees consists of rooted
ordered increasing trees with no restriction on the out-degrees of the nodes. A new vertex may be joined
to an existing vertex $v$ in exactly $\deg^{+}(v)+1$ positions, where $\deg^{+}(v)$
denotes the out-degree of node $v$. These $\deg^{+}(v)+1$ positions can be represented by external nodes (see Figure~\ref{STIRPfig0}).
Consequently, the total number of positions available for the $(n+1)$-st node being attached to a tree of size
$n$ is given by $\sum_{j=1}^{n}(\deg^{+}(j)+1)=2n-1$, independent of
the actual shape of the tree. There is exactly one tree
of size $1$, and  for $n\ge 1$, there are
$T_n=\prod_{\ell=1}^{n-1}(2\ell-1)=(2n-3)!!$ distinct
plane-oriented recursive trees of size $n$. From the formal description of increasing trees, since there are no restrictions on node out-degrees, we set the degree-weights to $\varphi_{\ell}=1$ for $\ell \ge 0$, so the
degree-weight generating function is given by $\varphi(t)=\sum_{\ell\ge 0}\varphi_{\ell}t^{\ell}
=\frac{1}{1-t}$. By solving the differential
equation~\eqref{eqnz1}, we get the exponential generating function $T(z)$ of plane-oriented recursive trees:
   \begin{equation*}
      T(z) = 1-\sqrt{1-2z}, \quad \text{and} \quad
      T_{n} =\prod_{\ell=1}^{n-1}(2\ell-1) = (2n-3)!!, \enspace \text{for} \; n \ge 1.
   \end{equation*}
Note that $T_{n+1}=(2n-1)!!$, equal to the number of $2$-Stirling permutations $\mathcal{Q}_n(2)$ of size $n$ and the number of ternary trees of size $n$.
\end{example}

\begin{remark}
Both the family $\mathcal{T}(k+1)$ of $(k+1)$-ary increasing trees and the family $\mathcal{P}$ of plane-oriented recursive trees can be generated according to tree evolution processes, as represented in Figure~\ref{STIRPfig0}. For a comprehensive discussion, we refer the interested reader to the work of Panholzer and Prodinger~\cite{PanPro2005+}.
\end{remark}

\subsection{Bijection of \texorpdfstring{$k$}{k}-Stirling permutations and \texorpdfstring{$(k+1)$}{(k+1)}-ary increasing trees}

Janson~\cite{Jan2008} has shown that $\mathcal{P}_{n+1} $ plane-oriented recursive trees of size $n+1$ are in bijection with $2$-Stirling permutations of size $n$, and Janson et al.~\cite{JanKuPa2008} gave a bijection between ternary increasing trees of size $n$ and plane-oriented recursive trees of size $n+1$.  More generally:

\begin{theorem}[Gessel~\cite{Park1994b}; see also~\cite{JanKuPa2008}]
\label{STIRPprop1T}
For $k\in\N$, the family of $(k+1)$-ary
increasing trees of size $n$ is in a natural bijection
with $k$-Stirling permutations of size $n$: $\mathcal{Q}_n(k)\cong \mathcal{T}_n(k+1)$.
\end{theorem}


As shown in~\cite{JanKuPa2008}, the bijection behind Theorem~\ref{STIRPprop1T} allows to study parameters
in $k$-Stirling permutations via the corresponding parameters in $(k+1)$-ary increasing trees.
The bijection is stated explicitly below.

\begin{bijection}
\label{TYPESbij}
Starting with a $(k+1)$-ary increasing tree $T$ of size $n$, we construct a $k$-Stirling permutation of size $n$ as follows.
We consider the representation of the tree $T$ where each labeled node has exactly $k+1$ ordered children, with unlabeled children marked by external nodes.  For each node labeled by an integer, we place $k$ copies of that integer between the $k+1$ children of the node.  We then collect these copies into a string through a contour walk around the tree starting at the root and going left.  (Equivalently, we construct the string by performing a depth-first walk, where we concatenate the integer $v$ to the string every time we visit the node labeled by $v$, except for the first and the $(k+1)$-st visit.)  This process results in a unique string of $k\cdot n$ integers $\tau=\tau_1 \tau_2 \ldots \tau_{kn}$, where each of the integers $1,\dots, n$ appears exactly $k$ times.  Since the tree $T$ is increasing, it guarantees that $\tau_i \le \tau_j \ge \tau_k$ whenever $\tau_i = \tau_k$ and $i<j<k$.  Therefore,  $\tau$ is a $k$-Stirling permutation of size $n$.

Note that through this process, the external nodes of the tree $T$ correspond to the gaps in the string $\tau$; adding the $(n+1)$-st labeled node to $T$ at one of its $kn+1$ external nodes corresponds to
inserting the $k$-tuple $(n+1)^k$
into the string $\tau$ at one of its $kn + 1$ gaps.

Conversely, starting with a $k$-Stirling permutation $\sigma$ of size $n$, we construct
the unique corresponding $(k+1)$-ary increasing tree through the following recursive procedure.
First, we decompose the permutation as $\sigma=\sigma_{1}1\sigma_2 1\dots\sigma_{k} 1 \sigma_{k+1}$.
The $\sigma_i$'s are either empty strings, or (with proper relabeling) they are $k$-Stirling permutations of size smaller than $n$.  We label the root node with integer $1$.  For $1\le i \le k+1$, we label the $i$-th child as follows: if $\sigma_i$ is empty, then the $i$'th child is an external node; else the $i$-th child is labeled by the smallest element in $\sigma_i$.  We repeat this process recursively for each non-empty $\sigma_i$.
\end{bijection}

\section{Local types in \texorpdfstring{$k$}{k}-Stirling permutations}\label{SEC localtypes}

Ordinary permutations can be classified according to four local types (e.g., see~\cite{Flajo1980, Conrad2006}) :

\begin{definition}
Let $\tau=\tau_1\dots\tau_n$ be a permutation of size $n$, and let $\tau_0=\tau_{n+1}=-\infty$.  Then for $1\le j\le n$, index $j$ is called a \emph{peak} if $\tau_{j-1}<\tau_j>\tau_{j+1}$, a \emph{valley} if $\tau_{j-1}>\tau_j<\tau_{j+1}$, a \emph{double rise} if $\tau_{j-1}<\tau_j<\tau_{j+1}$, and a \emph{double fall} if $\tau_{j-1}>\tau_j>\tau_{j+1}$.
\end{definition}

Note that sometimes the boundary condition
$\tau_{n+1}=+\infty$ is used~\cite{Conrad2006}; however, the
 $\tau_{n+1}=-\infty$ is more consistent with respect to
the bijection between ordinary permutations and binary increasing trees.
Since a permutation $\tau=\tau_1\dots\tau_n$ has a corresponding binary increasing tree~\cite{Flajo1980}, the correspondence between boundary condition of the local type of the index $j$ and the node type of the $j$-th labeled node in the binary increasing tree is as follows.

\smallskip

{\footnotesize
\begin{center}
\begin{tabular}{|c|cccc|}
\hline
Local type&Peak  & Valley & Double rise & Double Fall\\
\hline
Condition&$\tau_{j-1}<\tau_j>\tau_{j+1}$ &$\tau_{j-1}>\tau_j<\tau_{j+1} $& $\tau_{j-1}<\tau_j<\tau_{j+1}$ &
$\tau_{j-1}>\tau_j>\tau_{j+1}$ \\
\hline
Node type& Leaf & Double node & Right-branching node &Left-branching node.\\
\hline
\end{tabular}
\end{center}}

In extending the notion of local types to $k$-Stirling permutations, it is desirable to preserve the analogue of the correspondence between local types and node types.  Thus we first clarify the node types of (k+1)-ary increasing trees.
By definition of $(k+1)$-ary increasing trees, every node has exactly $k+1$ (labeled) positions for children.
Some of these positions may be occupied by labeled nodes, while others may be vacant, represented by external nodes.
We propose the following definition.
\begin{definition}
\label{DEFbij}
In a $(k+1)$-ary increasing tree $T$ of size $n$, the \emph{node type} of the node labeled $i$, for $1\le i \le n$, is defined as the string $G_i(T)=g_{i,1}\dots,g_{i,k+1}\in\{0,1\}^{k+1}$ of length $k+1$, where $g_{i,j}=1$ if the $j$-th child is a labeled node, and $g_{i,j}=0$ if the $j$-th child is vacant.
\end{definition}

In other words, the sequence $G_i(T)$ specifies which of the $(k+1)$ children are not empty.

\begin{example}
In the case $k=2$ of ternary increasing trees, we have $8=2^3$ different types of nodes.
The sequence 111 corresponds to a triple node, 101 to a (left,right)-branching node, 110 to a (left,center)-branching node, 011 to a (center,right)-branching node, 100 to a left-branching node, 010 to a center-branching node, 001 to a right-branching node, and 000 to a leaf, respectively. See Figure~\ref{TYPESfig1} for an illustration.
\end{example}

\begin{figure}[!htb]
\centering
\includegraphics[angle=0,scale=0.9]{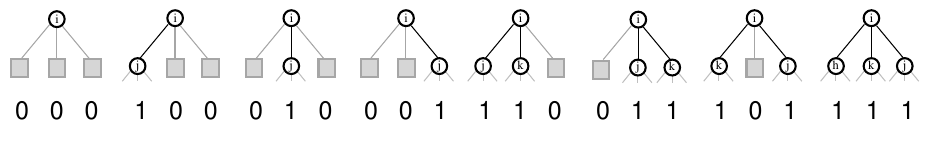}
\caption{The eight different node types in ternary increasing trees, assuming that $j,h,k>i\ge 1$.\label{TYPESfig1}}
\end{figure}

\begin{example}
In the case $k=1$ of binary increasing trees, we have $4=2^{2}$ different types of nodes.  The sequence 11 corresponds to a double node, the sequence 10 to a left-branching node, the sequence 01 to a right-branching node,
and the sequence 00 to a leaf.
\end{example}

This motivates an alternative definition for local types for ordinary permutations, which would lend to extension over the $k$-Stirling permutations.

\begin{definition}
\label{TYPESdef0}
Given an ordinary permutation $\tau=\tau_1\dots\tau_n\in\mathcal{S}_n$ and entry $i$, with $1\le i\le n$, let $j_i$ denote the index such that $\tau_{j_i}=i$. Set the boundary conditions to $\tau_0=\tau_{n+1}=-\infty$.  The local type $L_i(\tau)=\ell_{i,1}\ell_{i,2}$ of the entry $i$ in $\tau$ is a string of length $2$ with elements in $\{0,1\}$, defined as follows:

\begin{equation*}
\ell_{i,1} =
\begin{cases} 1 \quad \text{if}\,\,\tau_{j_{i}-1}>i,\\
0 \quad \text{otherwise};
\end{cases}
\quad\text{ and }\quad
\ell_{i,2} =
\begin{cases}1 \quad \text{if}\,\,\tau_{j_{i}+1}>i,\\
0 \quad \text{otherwise}.
\end{cases}
\end{equation*}

\end{definition}

The string $L_i(\tau)$ specifies which of the neighbors of $i$, going from left to right, are larger than $i$.  Thus $i$ is a peak if $L_i(\tau)=00$, a valley if $L_i(\tau)=11$, a double rise if  $L_i(\tau)=01$, and a double fall if $L_i(\tau)=10$.

\begin{example}
The ordinary permutation $\tau=2534716$ of size seven has the following local types:
$L_1(\tau)=11$, $L_2(\tau)=01$, $L_3(\tau)=11$, $L_4(\tau)=01$, $L_5(\tau)=00$, $L_6(\tau)=00$, and $L_7(\tau)=00$.
\end{example}

This new definition readily extends to the general case of $k$-Stirling permutation for $k\ge 1$.  We again define the local type of $i$ as a string of elements in $\{0,1\}$ that specifies which neighbors of $i$ in the permutation, going form left to right, are larger than $i$.

\begin{definition}
\label{TYPESdef1}
Given a $k$-Stirling permutation $\sigma=\sigma_1\sigma_2\dots\sigma_{kn}$ of size $n$, and an entry $i$, with $1\le i\le n$,
let $1\le j_{i,1}<\dots<j_{i,k}\le kn $ be the indices
such that $\sigma_{j_{i,h}}=i$. The
local type $L_i(\sigma)=\ell_{i,1}\dots \ell_{i,k+1}$ of the entry $i$ is a string
of length $k+1$, with elements in $\{0,1\}$, generated according to relative magnitudes of the instances of $i$ to their neighbors by the following rules (assuming boundary conditions
$\sigma_0=\sigma_{nk+1}=-\infty$):
\begin{equation*}
\ell_{i,1} =
\begin{cases} 1\quad \text{if}\,\,\sigma_{j_{i,1}-1}>i,\\
0 \quad \text{otherwise};
\end{cases}
\quad\text{ and }\quad
\ell_{i,j} =
\begin{cases}
1 \quad \text{if}\,\,\sigma_{j_{i,j-1}+1}>i,\\
0\quad \text{otherwise},
\end{cases}
\quad\text{ for } 1<j\le k+1.
\end{equation*}

\end{definition}

\begin{example}
The 3-Stirling permutation $\sigma=112233321445554666$ of size six has
the following local types $L_1(\sigma)=0011$, $L_2(\sigma)=0010$, $L_3(\sigma)=0000$, $L_4(\sigma)=0011$, $L_5(\sigma)=0000$, $L_6(\sigma)=0000$.
\end{example}

Since there are exactly $2^{k+1}$ different possible local types, we obtain the following result.
\begin{prop}
A $k$-Stirling permutation $\sigma=\sigma_1\sigma_2\dots\sigma_{kn}$ of size $n$ of the multiset
$\{1^{k}, 2^{k}, \dots, n^{k}\}$ can be classified according to $2^{k+1}$ different local types, with respect to the local rules in Definition~\ref{TYPESdef1}.
\end{prop}

\begin{remark}
The local types of $k$-Stirling permutations are closely related to the distribution of the number of $j$-ascents, $j$-descents and $j$-plateaux, as introduced by Janson et al.~\cite{JanKuPa2008}. These parameters generalize the standard notion of ascents, descents and plateaux for ordinary permutations. Let $2\le j\le k+1$: if $\ell_{i,j}=1$, then there is an occurrence of a $j$-ascent, implying a later $(j+1)$-descent. 
\end{remark}

It remains to prove that the local types of $k$-Stirling permutations correspond to the node types of $(k+1)$-ary increasing trees.

\begin{theorem}
\label{THMbij}
By Bijection~\ref{TYPESbij}, the local types $L_i(\sigma)$ in a $k$-Stirling permutation $\sigma=\sigma_1\dots\sigma_{kn}$ of size $n$ coincide with the node types $G_i(T)$ of the corresponding $(k+1)$-ary increasing trees $T$ of size $n$: $L_i(\sigma)=G_i(T)$ for $1\le i\le n$.
\end{theorem}

\begin{proof}
We use the bijection from Theorem~\ref{STIRPprop1T} between $k$-Stirling permutations and $(k+1)$-ary increasing trees, based on a depth-first walk. We start the depth-first walk at the root of a given $(k+1)$-ary increasing tree $T$ of size $n$ with node types $G_1(T),\ldots,G_n(T)$, and construct the corresponding $k$-Stirling permutation $\sigma=\sigma(T)$ of size $n$ according to Bijection~\ref{TYPESbij}.
Let $1\le j_{i,1}<\dots<j_{i,k}\le kn $ be the indices such that $\sigma_{j_{i,h}}=i$.
We show that the local order type $L_i(\sigma)=\ell_{i,1}\dots \ell_{i,k+1}$ equals the node type $G_i(T)=g_{i,1}\dots g_{i,k+1}$ of the node labeled $i$, for all $1\le i\le n$.

Let us consider the first position out of the $k+1$ positions for children of the node $i$.
If the first position is vacant, the first element of the node degree type is $g_{i,1}=0$. By Bijection~\ref{TYPESbij} the vacancy implies that $i=\sigma_{j_{i,1}}>\sigma_{j_{i,1}-1}=m$, and consequently the first element of the local type is $\ell_{i,1}=0$, since a smaller number $m<i$ must have been observed earlier according to the depth-first walk and the property that the tree is increasingly labeled. (If $i$ is the root, then we use the boundary conditions $m=-\infty$.)

On the other hand, if the first position is not vacant, then $g_{i,1}=1$. By the depth-first walk and the definition of increasing trees we have $i=\sigma_{j_{i,1}}<\sigma_{j_{i,1}-1}$, so $\ell_{i,1}=1$.

\begin{figure}[!htb]
\centering
\includegraphics[angle=0,scale=1]{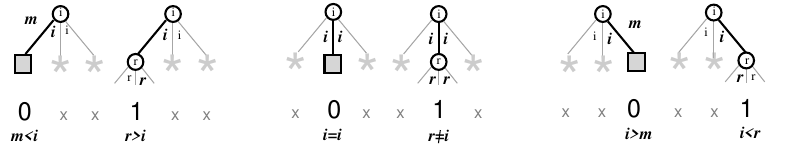}
\caption{A schematic representation of the correspondence between
local order types in $2$-Stirling permutations and local node types in ternary increasing trees according to
Bijection~\ref{TYPESbij}: the tree is traversed according to a depth-first walk, assuming that $r>i$ and
$m<i$.\label{TYPESfigproof}}
\end{figure}

An analogous rationale holds for the case of the $(k+1)$-st position of the node $i$.

Now let us consider the rest of the positions for the children of the node $i$.
For $1\le h \le k$, note that $g_{i,h}=0$ implies that the indices $j_{i,h}$
and $j_{i,h+1}$ satisfy $j_{i,h}+1=j_{i,h+1}$.  Consequently, $\sigma_{j_{i,h}+1}=i$, and thus $\ell_{i,h}=0$.
Conversely, if $g_{i,h}=1$, then by the construction in Bijection~\ref{TYPESbij} there is a substring between the $h$-th and $(h+1)$-st occurrences  of $i$, and this substring is composed of elements greater than $i$.  Consequently, $\sigma_{j_{i,h}+1}>i$, and thus $\ell_{i,h}=1$.

\end{proof}

\section{Path diagrams and local types of \texorpdfstring{$k$}{k}-Stirling permutations}\label{SEC PathLocTypes}
\subsection{Path diagrams and \texorpdfstring{$k$}{k}-Stirling permutations\label{TYPESsublocal}}

Ordinary permutations correspond to path diagrams related to Motzkin paths in their description via local types, as shown in the works of Fran\c{c}on and Viennot~\cite{Francon1979} and Flajolet~\cite{Flajo1980}.  In this section, we extend their ideas to give a bijection between $k$-Stirling permutations and a more general set of plane path diagrams called \L ukasiewicz paths.  


\begin{definition}
\label{def:L-path}
A \L ukasiewicz path is a path on the $x$-$y$ plane which starts at the origin, remains in the first quadrant (where $x\ge 0$ and $y\ge 0)$, ends on the $x$-axis, and which consists only of the following types of steps:

\begin{itemize}
    \item ``rise" steps $a^k$ from $(x,y)$ to $(x+1, y+k)$, for some non-negative integer $k$,
    \item ``fall" step $b$ from $(x,y)$ to $(x+1, y-1)$.
\end{itemize}

\end{definition}

There are many equivalent ways to represent a particular \L ukasiewicz path:

\begin{itemize}
    \item as a sequence of points in the $x$-$y$ plane, specified by their $x$-$y$ coordinates;
    \item as a sequence of the $y$-coordinates of those points (since the $x$-coordinates always start at 0 and increase by 1, according to the definition);
    \item as a word $L=L_1\dots L_n$ over the alphabet 
    $\mathcal{A}=\{b,a^0, a^1, a^2, \ldots\}$ corresponding to the sequence of steps of the path, beginning at the origin.
    
\end{itemize}

\begin{figure}[!htb]
    \centering
    \includegraphics[angle=0,scale=1]{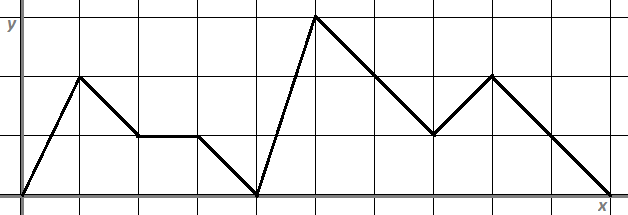}
    \caption{A \L ukasiewicz path.}
    \label{fig:L-path}
\end{figure}

The path in Figure~\ref{fig:L-path} is an example of a \L ukasiewicz path.  It can be represented as the following sequence of points in the $x$-$y$ plane: 
$$
(0,0), (1,2), (2,1), (3,1), (4,0), (5,3), (6,2), (7,1), (8,2), (9,1), (10,0).
$$
This path can be represented as the following sequence of the $y$-coordinates: 
$$
(0,2,1,1,0,3,2,1,2,1,0).
$$
This path can be represented as the following word over the alphabet

$\mathcal{A}=\{b,a^0, a^1, a^2, \ldots\}$ :
$$
a^2 b a^0 b a^3 b b a^1 b b .
$$

Keeping track of the $y$ coordinates is important when establishing correspondences between plane path diagrams and other combinatorial objects. We therefore define \emph{labeled} paths as the positive paths in which each step is indexed by the $y$ coordinate of
the point from which that step starts.  

\begin{notation}
\label{notation:labeled L-path}
Suppose a \L ukasiewicz path is represented by a word $L=L_1\dots L_n$ over the alphabet $\mathcal{A}=\{b,a^0, a^1, a^2, \ldots\}$  and a corresponding sequence $(0, y_1, \ldots, y_n)$ of the $y$-coordinates.  A \textit{labeled} version of the path has the same letters as the word $L=L_1\dots L_n$, with $j$-th letter having a subscript $y_{j-1}$ for each $j=1, \ldots, n$ (where $y_0=0$).

\end{notation}

For example, the labeled version of the path in Figure~\ref{fig:L-path} is
$$
a^2_0 b_2 a^0_1 b_1 a^3_0 b_3 b_2 a^1_1 b_2 b_1 .
$$

The methods which Flajolet presented in \cite{Flajo1980} of establishing a correspondence between a set of Motzkin paths and a desired set of some other combinatorial objects (e.g., permutations of $n$ objects) involve determining a scheme where, given a Motzkin path, one constructs a combinatorial object one path step at a time.  Such construction schemes allow for more than one choice of action at any one step, with the number of choices being a function of the $y$-coordinate of the point from which the step originates (called the \textit{possibility function}).  The bijection with the desired set of combinatorial objects is then with a pair $(M,p)$ where $M$ is a Motzkin path and $p$ is a \textit{possibility sequence} that keeps track of the choice one made at each step of construction
(see \cite{Flajo1980},~\cite{Francon1979}).

Let $\mathcal{A^+}$ be the alphabet of all the letters that could comprise a labeled \L ukasiewicz path:
\begin{equation*}
\mathcal{A^+}=\{b_0,b_1,b_2\dots\}\cup \bigcup_{\ell=0}^{\infty}\{a^\ell_0, a^\ell_1, a^\ell_2, \dots\}.
\end{equation*}

We adapt the definition of a system of path diagrams from~\cite{Flajo1980} to \L ukasiewicz paths:

\begin{definition}
A system of labeled \L ukasiewicz path diagrams is specified by a function $\pos: \mathcal{A^+}\to \N$ (called a \textit{possibility function}), which takes a letter in the alphabet $\mathcal{A^+}$ and assigns it a non-negative integer.   A \textit{path diagram} is a pair $(L, p)$, where $L=L_1\dots L_n$ is a labeled \L ukasiewicz path, and $p$ is a sequence of non-negative integers 
$(p_1,\dots, p_n)$ such that  $0\le p_j < \pos(L_j)$ for all $j=1,\ldots,n$.
\end{definition}

Now we are ready to state the connection between path diagrams and $k$-Stirling permutations, using their correspondence with $(k+1)$-ary increasing trees.

\begin{theorem}
\label{THMpath}
The set of $k$-Stirling permutations of size $n+1$ is in bijection with the system of labeled \L ukasiewicz path diagrams whose possibility function $\pos(.)$ is given by
\begin{equation*}
\pos(a^\ell_j)=\binom{k+1}{\ell+1}(j+1) \quad\text{ for }\quad 0\le \ell \le k, 
\quad\quad\quad \pos(b_j)=j+1,
\end{equation*}
and which has the following restrictions on the plane paths:
\begin{itemize}
    \item the plane paths are comprised of $n$ steps;
    \item the ``rise" steps are restricted to $a^0,\dots,a^k$.
\end{itemize}

\end{theorem}

\begin{remark}
For the case where $k=1$, Theorem~\ref{THMpath} reduces to the correspondence between Stirling permutations of size $n+1$ and Motzkin paths of length $n$, as presented by Fran\c{c}on and Viennot~\cite{Francon1979}. 
\end{remark}

\begin{proof}
By Theorem~\ref{STIRPprop1T}, the set of $k$-Stirling permutations of size $n+1$ is in bijection with the set of $(k+1)$-ary increasing trees
of size $n+1$.  It therefore suffices to establish a bijection between the system of labeled \L ukasiewicz path diagrams specified in Theorem~\ref{THMpath} and the set of $(k+1)$-ary increasing trees
of size $n+1$.

Given a path diagram $(L,p)$, where $L$ is labeled \L ukasiewicz path $L=L_1\ldots L_n$ and $p=(p_1,\ldots,p_n)$ is a sequence of non-negative integers with $0\le p_i < \pos(L_i)$, we will provide an algorithm for constructing a unique $(k+1)$-ary increasing tree with $n+1$ internal nodes.  At each step of the path, there will be a number of possible ways to carry out the instructions, which will correspond to the defined possibility function for that labeled step.  Thus each of the path diagrams induced by the path $L$ can be assigned a unique $(k+1)$-ary tree that can be constructed using $L$.

First, consider the following algorithm for constructing a $(k+1)$-ary tree using the path $L=L_1\ldots L_n$, which could result in different trees depending on the choices made during each step.  We begin with one placeholder for an internal node of the tree $T$.  For steps $1\le i \le n$: at $i$-th step of the path, choose one of the available placeholders, and replace it with a node labeled by $i$.  If $L_i=b_j$, there are no further instructions.  However, if $L_i=a^\ell_j$, then out of the $k+1$ possible child positions for the node, choose $\ell+1$ of them as placeholders for internal nodes.

Note that in each of the cases, the net change in the number of placeholders at the $i$-th step of the construction is equal to the change in the y-coordinate at the $i$-th step of the path.  Since we begin the construction with one placeholder, at the beginning of the $i$-th construction step the number of placeholders is one more than the y-coordinate of the point where the $i$-th step of the path begins. After the final $n$-th step, the y-coordinate the point where the last step ends is 0, so there is only one placeholder left.  To complete the tree, replace that placeholder with a node labeled by $n+1$.

Knowing the number of placeholders at each step allows us to compute the number of possible ways each construction step can be implemented, and thus determine the appropriate possibility function.  If $L_i=b_j$, there are $j+1$ possible placeholders to choose from for the new node labeled by $i$, so $\pos(b_j)=j+1$. If $L_i=a^\ell_j$, the number of possibilities is $\pos(a^\ell_j)=
\binom{k+1}{\ell+1}(j+1)$, taking into account both the number of existing placeholders to choose from, and the possibilities for choosing a new placeholder out of the node's $\ell+1$ possible child positions.

To establish a unique correspondence between the path diagram $(L,p)$, we will induce an ordering on the internal nodes (depth-first, with the child nodes ordered from left to right) and on the node types (lexicographical ordering).  The algorithm for constructing a unique $(k+1)$-ary tree $T$ using this path diagram is then as follows.  

As before, we begin with one placeholder for an internal node of the tree $T$. For steps $1\le i \le n$: 

\begin{itemize}
    \item If the $i$-th step of the path is a ``fall" step $L_i=b_j$: using depth-first order, choose the $(p_i+1)$-st available placeholder.  Replace it with a node labeled by $i$ that has $k+1$ external child nodes (so the chosen node becomes a leaf with respect to internal nodes of the tree).
    \item If the $i$-th step of the path is a ``rise" step $L_i=a^\ell_j$: Let $s$ be the number of integer times $\binom{k+1}{\ell+1}$ goes into $p_i$, and let $t$ be the remainder of $p_i$ modulo $\binom{k+1}{\ell+1}$, such that $p_i=s\cdot \binom{k+1}{\ell+1}+t$. Using depth-first order, choose the $(s+1)$-st available placeholder.  Using the lexicographical ordering on node types with $\ell+1$ external child nodes out of a total of $(k+1)$ child nodes, choose the $(t+1)$-st local type. Replace the chosen placeholder with a node of this node type, with all the external node children labeled with placeholders, and label the chosen node by $i$.
\end{itemize}

After $n$ steps, there will be a single available placeholder remaining.  To complete the tree, replace that placeholder with a node labeled by $n+1$ that has $k+1$ external child nodes.
(Example~\ref{Ex-for-THMpath} goes through a specific example of this algorithm. ) 

\end{proof}

\begin{example}\label{Ex-for-THMpath}
In this example, we consider the case $k=2$, which corresponds to ternary increasing trees and to Stirling permutations.
Figure~\ref{TYPESfig2} illustrates the steps of the algorithm specified in the proof of Theorem~\ref{THMpath} for constructing a ternary tree with seven labeled nodes using the path diagram $(L,p)$ where 
$L=a^2_0 a^1_2 b_3 b_2 a^0_1 b_1$ and $p=(0,1,3,0,3,1)$.

\begin{figure}[htb]
\centering
\includegraphics[angle=0,scale=0.9]{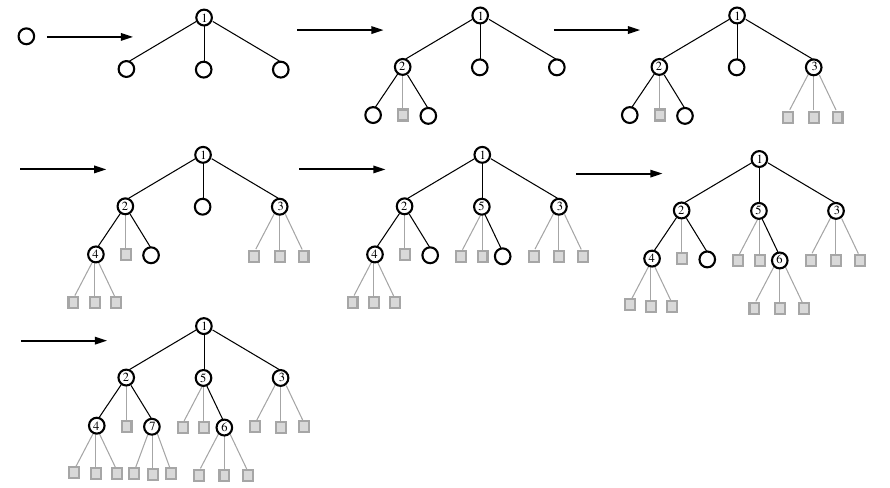}
\caption{Construction of a ternary increasing tree. \label{TYPESfig2}}
\end{figure}

The first arrow indicates the step corresponding to the first path step $L_1 = a^2_0$; the first (and at this stage, only) available placeholder gets replaced with a node of type 111 and labeled ``1". (One can check that $s=t=0$ at this step, which correspond with the choice of the first available placeholder and the first node type.)

The second arrow indicates the step corresponding to the second path step $L_2=a^1_2$; $p_2=1$, and $\binom{2+1}{1+1}=3$, so $s=0$ and $t=1$.  Thus the first available placeholder gets chosen (through depth-first method with left-to-right order) to get the label ``2".  Since by lexicographic order $011<101<110$, the second node type 101 gets chosen, with the internal child nodes getting the placeholders.

The third arrow indicates the step corresponding to the third path step $L_3=b_3$; $p_3=3$, so the fourth available placeholder gets chosen (through depth-first method with left-to-right order) to get the label ``3", with three external child nodes.

The algorithm continues until the seventh arrow, where the final available placeholder is gets the label ``7" and gets three external child nodes.

By using the Bijection~\ref{TYPESbij} between $(k+1)$-ary increasing trees and $k$-Stirling permutations for $k=2$,
we immediately obtain the Stirling permutation $\sigma$ of size seven that corresponds to the tree constructed in Figure~\ref{TYPESfig2}: $\sigma=44227715566133$.
Note that we can also directly construct the Stirling permutation, since the local types of the outdegree of the nodes
in the ternary increasing tree correspond to the local types of the numbers in the permutation. One may think of this procedure as some kind of ``flattening
of the tree to a line'' (compare with the sequence of trees in Figure~\ref{TYPESfig2}):
\begin{equation*}
\begin{split}
\circ& \to \circ1\circ1\circ \to \circ22\circ1\circ1\circ \to \circ22\circ1\circ133
\to 4422\circ1\circ133 \to 4422\circ155\circ133 \\
&\to 442277155\circ133 \to 44227715566133.
\end{split}
\end{equation*}
\end{example}

\subsection{Generating function of local types}\label{SEC ogf local types}

Flajolet~\cite{Flajo1980} used the correspondence between path
diagrams and formal power series to obtain continued
fraction representations of the generating functions of many
parameters in ordinary permutations. In the context
of permutations and binary increasing trees, he derived a continued fraction representation of the generating
function of local types in permutations, and equivalently of node
types in binary increasing trees. The path diagram
representation of $k$-Stirling permutations and $(k+1)$-ary
increasing trees similarly allows us to obtain a branched continued fraction representation of the generating function of the local types and
of node types. 

First we have to recall relevant definitions from Flajolet's work~\cite{Flajo1980} concerning formal power series.
Let $C\lAngle X\rAngle$ denote the monoid algebra of formal power series $s=\sum_{u\in X^{\ast}}s_u\cdot u$ on the set of
non-commutative variables (alphabet) $X$ with coefficients in the field of complex
numbers, with sums and Cauchy product defined in the usual way
\begin{equation*}
s+t=\sum_{u\in X^{\ast}}(s_u+t_u)\cdot u,\qquad s\cdot t=\sum_{u\in X^{\ast}}\bigg(\sum_{vw=u}s_vt_w\bigg)\cdot u.
\end{equation*}
In order to define the convergence of a series, one introduces
the valuation of a series $\val(s)$, defined by
\begin{equation*}
\val(s)=\min\{|u|: s_u\neq 0\},
\end{equation*}
where $|u|$ denotes the length of the word $u\in X^*$.
A sequence of elements $(s_n)_{n\in\N}$, $s_n\in C\lAngle X\rAngle$, converges to a limit $s\in C\lAngle X\rAngle$ if
\begin{equation*}
\lim_{n\to\infty}\val(s-s_n)=\infty.
\end{equation*}
Multiplicative inverses exist for series having a constant term different from
zero; for example $(1-u)^{-1}=\sum_{\ell\ge 0}u^{\ell}$, where $(1-u)^{-1}$ is an inverse of $u$ so long as $\val(u)>0$.
Note that we will subsequently use the notation $(u|v)/w=uw^{-1}v$.
The characteristic series $\chara(S)$ of $S\subset X^*$ is defined as
\begin{equation*}
\chara(S)=\sum_{u\in S}u.
\end{equation*}
Finally, following~\cite{Flajo1980} we use for subsets $E, F$ of $X^*$ the alternative notations
$E + F $ for the union $E\cup F$, $E\cdot F$ for the extension to sets of the catenation operation on words, and let
$E^* = \epsilon + E + E\cdot E + E\cdot E\cdot E + \dots$,  with $\epsilon$ denoting the empty word.
Moreover, we will use a lemma (Lemma~1 of Flajolet~\cite{Flajo1980}), which allows to translate operations on sets of words into
corresponding operations on series, provided certain non-ambiguity conditions are satisfied.

\smallskip

\begin{lemma}
\label{Flaj}
Let $E$, $F$ be subsets of $X^*$. Then
\begin{enumerate}
\item $\chara(E+F)=\chara(E)+\chara(F)$ provided $E\cap F=\emptyset$,
\item $\chara(E\cdot F)=\chara(E)\cdot\chara(F)$ provided that $E\cdot F$ has the unique factorization property, $\forall u,u'\in E$ $\forall v,v'\in F$ $uv=u'v'$ implies $u=u'$ and $v=v'$,
\item $\chara(E^*)=(1-\chara(E))^{-1}$ provided the following two condition hold: $E^j\cap E^k=\emptyset$ $\forall j,k $ with $j\neq k$, each $E^k$ has the unique factorization property.
\end{enumerate}
\end{lemma}
With the help of Lemma~\ref{Flaj}, one can translate operations on sets of words into corresponding operations on series, provided that the non-ambiguity conditions are
satisfied.

Let $C^{[h]}_i=C^{[h]}_i(k)$ be defined as the characteristic series of all labeled paths with steps given by $a^0, a^1,\dots,a^k,b$ starting and ending at y-coordinate $i$, with $i\ge 0$, never going below the y-coordinate $i$ and above the y-coordinate $i+h$, with $h\ge 0$.
We assume the formal convention that $C^{[h]}_i=0$ if $h<0$. Moreover, let $C^{[h]}=C^{[h]}_0$. We introduce the notation

$$\left<C_i^{[h]}\right>_1:=(a^1_i|b_{i+1}) C^{[h-1]}_{i+1}, \quad\quad \left<C_i^{[h]}\right>_2:=((a^2_i|b_{i+2})\cdot C^{[h-2]}_{i+2}|b_{i+1})\cdot C^{[h-1]}_{i+1},$$

and in general for integer $1\le \ell\le k$ let $\left<C_i^{[h]}\right>_{\ell}$ be defined by
\begin{equation*}
\left<C_i^{[h]}\right>_{\ell}= (\dots((a^\ell_i|b_{i+\ell}) C^{[h-\ell]}_{i+\ell}|b_{i+\ell-1}) C^{[h-(\ell-1)]}_{i+\ell-1}\dots |b_{i+1}) C^{[h-1]}_{i+1}.
\end{equation*}

\begin{prop}
\label{PROPfrac}
The characteristic series $C^{[h]}_i=C^{[h]}_i(k)$ of all labeled paths with steps given by $a^0, a^1,\dots,a^k,b$ starting and ending at y-coordinate $i$, with $i\ge 0$, never going below y-coordinate $i$ and above y-coordinate $i+h$, with $h\ge 0$, satisfies
\begin{equation*}
C^{[h]}_i=\frac{1}{1-a^0_i- \sum_{\ell=1}^{k} \left<C_i^{[h]}\right>_{\ell}}.
\end{equation*}
The double sequence $(C^{[h]}_i)_{i,h\ge 0}$ converges for $h\to\infty$. Its limit $(C_i)_{i\ge 0}$ is given as follows:
\begin{equation*}
C_i=\frac{1}{1-a^0_i- \sum_{\ell=1}^{k} \Big<C_i\Big>_{\ell}}.
\end{equation*}

In particular, $C=C_0$ equals the characteristic sequence of all labeled paths $\mathcal{P}$, starting and ending at the $x$-axis, never going below the $x$-axis, with steps given by $a^0, a^1,\dots,a^k,b$.
\end{prop}

\begin{remark}
The case $k=1$, treated by Flajolet~\cite{Flajo1980}, corresponds to binary increasing trees and ordinary permutations.
\end{remark}

\begin{proof}
For the sake of simplicity we only present the proof of the special case $k=2$, corresponding
to Stirling permutations and ternary increasing trees.
We prove that
\begin{equation*}
C^{[h]}_0= \frac{1}{1-a^0_0- \sum_{\ell=1}^{2} \Big<C_0^{[h]}\Big>_{\ell}}=  \frac{1}{1-a^0_0-(a^1_0|b_{1}) C^{[h-1]}_{1} - ((a^2_0|b_{2}) C^{[h-2]}_{2}|b_{1}) C^{[h-1]}_{1}}
\end{equation*}
equals the characteristic series of the set $\mathcal{P}^{[h]}$ of all labeled \L ukasiewicz paths with steps
$a^0, a^1, a^2, b$, starting and ending at y-coordinate zero with the y-coordinate bounded by 0 and some positive integer $h$.
More generally, for $i\ge 0$
\begin{equation*}
C^{[h]}_i=\frac{1}{1-a^0_i- \sum_{\ell=1}^{2} \left<C_i^{[h]}\right>_{\ell}}=\frac{1}{1-a^0_i-(a^1_i|b_{i+1}) C^{[h-1]}_{i+1} - ((a^2_i|b_{i+2}) C^{[h-2]}_{i+2}|b_{i+1}) C^{[h-1]}_{i+1}}
\end{equation*}
equals the characteristic series of all labeled \L ukasiewicz paths starting and ending at y-coordinate $i$ with y-coordinates bounded by $i$ and $i+h$ for some positive integer $h$.

Note that by our previous notation $(u|v)/w=uw^{-1}v$ and $(1-u)^{-1}=\sum_{\ell\ge 0}u^{\ell}$ regarding inverse series, we have for instance
\begin{equation*}
\begin{split}
(a^1_i|&b_{i+1}) C^{[h-1]}_{i+1}=a^1_i\big(C^{[h-1]}_{i+1}\big)^{-1}b_{i+1} \\
	&= a^1_i\sum_{\ell\ge 0}\Big(c_{i+1}+(a^1_{i+1}|b_{i+2}) C^{[h-2]}_{i+2} + ((a^2_{i+1}|b_{i+3}) C^{[h-3]}_{i+3}|b_{i+2}) C^{[h-2]}_{i+2}b_{i+2}\Big)^{\ell}b_{i+1}.
\end{split}
\end{equation*}

For the first few values of $h=1,2,3$ we obtain
\begin{equation*}
\begin{split}
\mathcal{P}^{[0]}&=(a^0_0)^*\\
\mathcal{P}^{[1]}&=(a^0_0 + a^1_0 (a^0_1)^*b_1)^*\\
\mathcal{P}^{[2]}&=(a^0_0 + a^1_0 (a^0_1+a^1_1 (a^0_2)^*b_2)^*b_1 
+a^2_0 (a^0_2)^*b_2(a^0_1+a^1_1 (a^0_2)^*b_2)^*b_1)^*\\
\mathcal{P}^{[3]}&=\Big(a^0_0 + a^1_0\big(a^0_1+a^1_1 (a^0_2+a^1_2 (a^0_3)^*b_3)^*b_2
+a^2_1 (a^0_3)^*b_3(a^0_2+a^1_2 (a^0_3)^*b_3)^*b_2\big)^*b_1 \\
&+a^2_0 (a^0_2+a^1_2 (a^0_3)^*b_3)^*b_2\big(a^0_1+a^1_1 (a^0_2+a^1_2 (a^0_3)^*b_3)^*b_2 \\
&\quad\quad\quad +a^2_1 (a^0_3)^*b_3(a^0_2+a^1_2 (a^0_3)^*b_3)^*b_2\big)^*b_1\Big)^*.
\end{split}
\end{equation*}

In order to simplify the recursive description of $\mathcal{P}^{[h]}$ we introduce the refined sets $\mathcal{P}^{[h]}_i$ consisting of the paths starting and ending at y-coordinate $i$, with y-coordinates for all steps bounded by $i$ and $i+h$ for some positive integer $h$, where $\mathcal{P}^{[h]}_0=\mathcal{P}^{[h]}$.  Note that $\mathcal{P}^{[h]}_i$
can easily be obtained from $\mathcal{P}^{[h]}_0=\mathcal{P}^{[h]}$ by shifting the index encoding the y-coordinate level by $i$.

We can write $\mathcal{P}^{[h]}$ in the following way.
\begin{equation*}
\begin{split}
\mathcal{P}^{[0]}&=(a^0_0)^*\\
\mathcal{P}^{[1]}&=(a^0_0 + a^1_0\mathcal{P}^{[0]}_{1} b_1)^*\\
\mathcal{P}^{[2]}&=(a^0_0 + a^1_0\mathcal{P}^{[1]}_{1} b_1         
                +a^2_0\mathcal{P}^{[0]}_{2}b_2\mathcal{P}^{[1]}_{1}b_1)^*\\
\mathcal{P}^{[3]}&=\Big(a^0_0 + a^1_0\mathcal{P}^{[2]}_{1}b_1
                +a^2_0\mathcal{P}^{[1]}_{2}b_2\mathcal{P}^{[2]}_{1}b_1
                +a^3_0\mathcal{P}^{[0]}_{3}b_3\mathcal{P}^{[1]}_{2}b_2\mathcal{P}^{[2]}_{1}b_1\Big)^*.
\end{split}
\end{equation*}
By induction one can prove the following unambiguous description of $\mathcal{P}^{[h]}$.
\begin{equation*}
\begin{split}
\mathcal{P}^{[h]}&=\Big(a^0_0 + a^1_0\mathcal{P}^{[h-1]}_{1}b_1 +a^2_0\mathcal{P}^{[h-2]}_{2}b_2\mathcal{P}^{[h-1]}_{1}b_1 + \cdots
+a^h_0\mathcal{P}^{[0]}_h b_h \ldots \mathcal{P}^{[h-1]}_{1}b_1\Big)^*.
\end{split}
\end{equation*}
More generally, we have
\begin{equation*}
\begin{split}
\mathcal{P}^{[h]}_i&=\Big(a^0_i + a^1_i\mathcal{P}^{[h-1]}_{i+1}b_{i+1} +a^2_i\mathcal{P}^{[h-2]}_{i+2}b_{i+2}\mathcal{P}^{[h-1]}_{i+1}b_{i+1} + \cdots
+a^h_i\mathcal{P}^{[0]}_{i+h} b_{i+h} \ldots \mathcal{P}^{[h-1]}_{i+1}b_{i+1}\Big)^*.
\end{split}
\end{equation*}
Since $\mathcal{P}^{[h]}_i$ is obtained from $\mathcal{P}^{[h]}_{0}=\mathcal{P}^{[h]}$ by an index shift, we recursively obtain the stated description
of $C_0^{[h]}$ by replacing the operations $+,\cdot,*$ on the sets of words by the series operations $+,|$, and inverse. Moreover, the characteristic series of the refined sets $\mathcal{P}^{[h]}_i$ is simply given by $C_i^{[h]}$.
One observes the inclusion
\begin{equation*}
\mathcal{P}^{[0]}\subset \mathcal{P}^{[1]}\subset \mathcal{P}^{[2]}\subset \dots \subset \mathcal{P},
\end{equation*}
or more generally
\begin{equation*}
\mathcal{P}^{[0]}_i\subset\mathcal{P}^{[1]}_i\subset \mathcal{P}^{[2]}_i\subset \dots \subset \mathcal{P}_{i}\quad i\ge 0.
\end{equation*}

A \L ukasiewicz path which starts at the x-axis and whose highest y-coordinate is $h$ must have length greater than or equal to $\big\lceil \frac{h}{k}\big\rceil$, thus
\begin{equation*}
\val\Big(C_i-C_i^{[h-1]}\Big)\ge \Big\lceil \frac{h}{k} \Big\rceil,
\end{equation*}
and consequently
\begin{equation*}
\lim_{h\to\infty}C_i^{[h]}=C_i.
\end{equation*}
The idea of the proof of the general case $k>2$ is similar.
\end{proof}

Subsequently, we will enumerate $k$-Stirling permutations, keeping track of the $2^{k+1}$ different local types.
Recall that each local type is a string of length $k+1$ over the alphabet $\{0,1\}$.  For our purposes, we arrange the local types first by the number of 1's, and then by lexicographic order.  For a given $k$-Stirling permutation, we will denote by $m_{i,j}$ the number of instances of the local type with $i$ 1's which are in the $j$-th position in the lexicographic order.

Let $P_{\mathbf{m}_0,\dots,\mathbf{m}_{k+1}}$ denote the number of $k$-Stirling permutations whose local types are
counted according to $\mathbf{m}_i=(m_{i,1},\dots, m_{i,\binom{k+1}{i}})$, with $0\le i \le k+1$.  To keep track of these counts, we will use the corresponding variables  $\mathbf{z}_i=(z_{i,1},\dots, z_{i,\binom{k+1}{i}})$, with $0\le i \le k+1$.  Note that $\mathbf{z}_0 = (z_{0,1})$ is the variable that keeps track of the local type with all zeros, and it would be unambiguous to refer to either $\mathbf{z}_0$ or $z_{0,1}$.

The generating function $P(\mathbf{z}_0,\dots,\mathbf{z}_{k+1},t)$ of $k$-Stirling permutations with respect to the $2^{k+1}$ local types is defined by
\begin{equation*}
P(\mathbf{z}_0,\dots,\mathbf{z}_{k+1},t)= \sum_{\mathbf{m}_0,\dots,\mathbf{m}_{k+1}}P_{\mathbf{m}_0,\dots,\mathbf{m}_{k+1}}\mathbf{\mathbf{z}_0}^{\mathbf{m}_0}
\dots\mathbf{\mathbf{z}_{k+1}}^{\mathbf{m}_{k+1}}t^n,
\end{equation*}
where $n=\sum_{i=1}^{k+1}\sum_{j=1}^{ \binom{k+1}{i}  }m_{i,j}$ is the length of the $k$-Stirling permutations.  

Now we can state the main result of this section: a branched continued fraction representation of the generating function
of local types in $k$-Stirling permutations.

\begin{theorem}\label{THM ogf k-striling}
The generating function $P(\mathbf{z}_0,\dots,\mathbf{z}_{k+1})$ of $k$-Stirling permutations, or equivalently $(k+1)$-ary increasing trees, with respect to the $2^{k+1}$ local types,
\begin{equation*}
P(\mathbf{z}_0,\dots,\mathbf{z}_{k+1},t)= \sum_{\mathbf{m}_0,\dots,\mathbf{m}_{k+1}}P_{\mathbf{m}_0,\dots,\mathbf{m}_{k+1}}\mathbf{\mathbf{z}_0}^{\mathbf{m}_0}
\dots\mathbf{\mathbf{z}_{k+1}}^{\mathbf{m}_{k+1}}t^n
\end{equation*}
is given by the branching continued fraction
\begin{equation*}
\frac{\mathbf{z}_0 t}{
1-1\cdot t\sum_{i=1}^{k+1}z_{1,i}
-\frac{1\cdot 2 \cdot t^2 \mathbf{z}_0 \sum_{i=1}^{\binom{k+1}{2}}z_{2,i}}
      {1-2\cdot t \sum_{i=1}^{k+1}z_{1,i} -\frac{2\cdot 3 \cdot t^2 \mathbf{z}_0
                                                  \sum_{i=1}^{\binom{k+1}{2}}z_{2,\ell}}
                                                {\dots} -\dots}
      -\dots -\frac{(k+1)!t^{k+1}\mathbf{z}_0^k z_{k+1,1}}{\dots}}.
\end{equation*}
\end{theorem}

\begin{coroll}\label{Corollary}
An expansion of the generating function $\sum_{n\ge 0}k^{n}\frac{\Gamma(n+1+\frac1k)}{\Gamma(\frac1k)}t^n$ is obtained from the generating function
$P(\mathbf{z}_0,\dots,\mathbf{z}_{k+1},t)$ by setting $\mathbf{z}_{\ell}=(1,\dots,1)$, $0\le \ell\le k+1$, and dividing by $t$.
In particular, we obtain for $k=2$ the identity
\begin{equation*}
\begin{split}
\sum_{n\ge 0}(2n+1)!!\,t^n
&= \frac{1}{1 - 1\cdot \binom{3}{1}t - \frac{1\cdot 2\cdot \binom{3}{2} t^2}{1-2\cdot \binom{3}{1}t -\frac{2\cdot 3\cdot \binom{3}{2} t^2}{1-3\cdot \binom{3}1 t\dots}-\frac{2\cdot 1\cdot 2\cdot 3\cdot \binom{3}{3} t^3}{1-4\cdot \binom{3}1 t\dots}} - \frac{1\cdot 2\cdot 3\cdot \binom{3}{3} t^3}{\big(1-3\cdot \binom31 t\dots\big)\big(1-2\cdot\binom{3}{1}t\dots\big)}}\\
&= 1 + 3t  +15t^2+105t^3 +945t^4+\dots
\end{split}
\end{equation*}
\end{coroll}

\begin{remark}
Below each fraction bar in the branched continued fraction representation of the formal power series there are $k+1$ terms, starting with $1$.
As mentioned earlier the case $k=1$ is a result of Flajolet~\cite{Flajo1980}.
\end{remark}

\begin{remark}
P\'etr\'eolle, Sokal, and Zhu~\cite{PSZ2018} demonstrate that branched continued fractions do not provide unique representations of power series.  They show how the same infinite series presented in Collorary~\ref{Corollary} can have at least three different branched continued fraction representations.  
\end{remark}

\begin{proof}
According to Theorem~\ref{THMbij}, we can use the same representation for the local types of $k$-Stirling permutations and node types of $(k+1)$-ary trees.  In particular, the local types with $i$ 1's correspond to the node types with $i$ internal children.

Recall the construction from the proof of Theorem~\ref{THMpath} of a $(k+1)$-ary tree of size $n+1$, given a path diagram $(L,p)$ with $L$ a labeled \L ukasiewicz path of length $n$.  The ``rise" step $a^\ell_j$ corresponds to a node with $\ell+1$ internal children, and the ``fall" step $b_j$ corresponds to a leaf node with respect to internal children.  At the end of the construction one more leaf node was created.

Taking these correspondences into account, and considering the number of possibilities described by the construction, we define the morphism
$\mu:  C\lAngle X\rAngle\to C\,\lbrakk\, \mathbf{z}\,\rbrakk$ by

\begin{equation*}
\begin{split}
\mu(a^\ell_j) &= (j+1)t \sum_{i=1}^{  \binom{k+1}{\ell+1}} z_{\ell+1,i}\quad\text{ for } \quad 0\le \ell \le k, \\
\mu(b_j)        &= (j+1)t \mathbf{z}_0.
\end{split}
\end{equation*}

Taking into account the creation of one more leaf node at the end of the construction, we get the generating function of $(k+1)$-ary trees with respect to the node types from the characteristic series $C$ of labeled paths:
$P(\mathbf{z}_0,\dots,\mathbf{z}_{k+1},t)= z_0 t \mu(C)$.

\end{proof}

\section{Other branched continued fractions for $k$-Stirling permutations \label{SEC Stirling}}
In the following we present alternative representations of the generating function of $k$-Stirling permutations. 


\subsection{Classical Stirling permutations}
In the special case $k=2$ Janson~\cite{Jan2008} showed that the class of $2$-Stirling permutations of size $n$ is in bijection with the class of plane-oriented recursive trees of size $n+1$.
A bijection between ternary increasing trees of size $n$ and plane-oriented recursive trees of size $n+1$ was given in~\cite{JanKuPa2008}. We will provide a bijection between path diagrams with \L ukasiewicz paths with unrestricted number of ``rise" steps, plane-oriented recursive trees and Stirling permutations.

\begin{theorem}
\label{THMpath2}
The class of plane-oriented recursive trees of size $n+1$ is in bijection with the system of labeled \L ukasiewicz path diagrams whose possibility function $\pos(.)$ is given by
\begin{equation*}
\pos(a^\ell_j)=j+1 \quad \text{ for } \quad 0\le \ell, 
\quad \text{ and } \quad \pos(b_j)=j+1,
\end{equation*}
and whose plane paths are comprised of $n$ steps.

\end{theorem}

\begin{proof}
The proof is similar to that of Theorem~\ref{THMpath}.
First, consider the following algorithm for constructing a plane-oriented recursive tree of size $n+1$ using the path $L=L_1\ldots L_n$, which could result in different trees depending on the choices made during each step.
We begin the construction with one placeholder for an internal node.  Then for $1\le i \le n$, at the $i$-th step:
replace any existing placeholder with the node labeled by $i$;
if $L_i=a^\ell_j$, create $\ell+1$ children for the node $i$, each labeled by a placeholder; and if $L_i=b_j$, the node $i$ remains a leaf node.  

Note that in each case, the net number of placeholders changes by the same amount as the height of the path; since the construction began with one placeholder, at the beginning of the $k$-th construction step the number of placeholders is one more than the y-coordinate where the $k$-th step of the path $L$ begins.  The number of possibilities for carrying out the $k$-th step is $j+1$ whether $L_k$ is the ``rise" step $a^\ell_j$ or the ``fall" step $b_j$, which defines the corresponding possibility function for the path diagrams.  At the end of the $n$-th step, there is one more placeholder, which we replace by a leaf node labeled by $n+1$.

To establish a unique correspondence between the path diagram $(L,p)$ and a plane-oriented recursive tree of size $n+1$, we will induce an ordering on the internal nodes: depth-first, with the child nodes ordered from left to right.  The algorithm for constructing a unique plane-oriented recursive tree $T$ of size $n+1$ using this path diagram is then as follows.  

As before, we begin with one placeholder for an internal node of the tree $T$. For steps $1\le i \le n$: 

\begin{itemize}
    \item If the $i$-th step of the path is a ``fall" step $L_i=b_j$: using depth-first order, choose the $(p_i+1)$-st available placeholder.  Replace it with a node labeled by $i$, and leave the node as a leaf.
    \item If the $i$-th step of the path is a ``rise" step $L_i=a^\ell_j$: using depth-first order, choose the $(p_i+1)$-st available placeholder.  Replace it with a node labeled by $i$ which has $\ell+1$ children, each labeled by a placeholder.
   
\end{itemize}

After $n$ steps, there will be a single available placeholder remaining.  To complete the tree, replace that placeholder with a leaf node labeled by $n+1$.

\end{proof}

Let $\mathcal{F}_n$ denote the the system of labeled \L ukasiewicz path diagrams specified in Theorem~\ref{THMpath2}, and let $f(t)=\sum_{n\ge 0}|\mathcal{F}_n| t^n$ be the ordinary generating function of these path diagrams.  Let $\mathcal{G}_n$ denote the system of labeled \L ukasiewicz path diagrams whose plane paths are restricted to the steps $a^0$, $a^1$, $a^2$, and $b$, and whose possibility function is defined by $\pos(a^0_j)=\pos(a^1_j)=3(j+1)$ and $\pos(a^2_j)=\pos(b_j)=j+1$. 
Let $g(t)=\sum_{n\ge 0}|\mathcal{G}_n| t^n$ be the ordinary generating function of $\mathcal{G}_n$.

The set $\mathcal{S}_n$ of Stirling permutations of size $n$ is in bijection with the set of path diagrams $\mathcal{G}_{n-1}$ by Theorem~\ref{THMpath}.  Also, since $\mathcal{S}_n$ is in bijection with the set of plane-oriented recursive paths of size $n+1$ (see Janson~\cite{Jan2008}), which in turn is in bijection with the set of path diagrams $\mathcal{F}_n$ by Theorem~\ref{THMpath2}, we get the identity $1 + S(t)= f(t) = 1+t g(t)$.

\begin{theorem}\label{THM2stirling}
	Let $S_n=(2n-1)!!$ denote the number of Stirling permutations of size $n$ and $S(t)=\sum_{n\ge 1}S_n t^n$ its the generating function. Then $1 + S(t)= f(t) = 1+t g(t)$, and can therefore be represented by two different continuous fraction types:
	\begin{equation}
	\begin{split}
	1+S(t)&=\frac{1}
	{1- t-
		\frac{2\cdot 1\cdot t^2 }
		{1- 2\cdot t - \frac{3\cdot 2\cdot t^2}{1- 3\cdot t -\cdots}
		}-
		\frac{3\cdot 2\cdot 1\cdot t^3}
		{\big(1-2\cdot t - \frac{3\cdot 2\cdot t^2}{1- 3\cdot t -\cdots}-\cdots \big)
			\big(1- 3\cdot t - \frac{4\cdot 3\cdot t^2}{1-4\cdot t -\cdots}-\cdots \big)
		}-
		\frac{4\cdot3\cdot 2\cdot 1\cdot t^3}{(\cdots)(\cdots)(\cdots)}
		-\cdots
	}\\
	&= 1+\frac{t}{1 - 1\cdot \binom{3}{1}t - \frac{1\cdot 2\cdot \binom{3}{2} t^2}{1-2\cdot \binom{3}{1}t -\frac{2\cdot 3\cdot \binom{3}{2} t^2}{1-3\cdot \binom{3}1 t\dots}-\frac{2\cdot 1\cdot 2\cdot 3\cdot \binom{3}{3} t^3}{1-4\cdot \binom{3}1 t\dots}} - \frac{1\cdot 2\cdot 3\cdot \binom{3}{3} t^3}{\big(1-3\cdot \binom31 t\dots\big)\big(1-2\cdot\binom{3}{1}t\dots\big)}}\\
	\end{split}
	\end{equation}
	
\end{theorem}

\begin{proof}
	The set $\mathcal{S}_n$ of Stirling permutations of size $n$ is in bijection with the set of path diagrams $\mathcal{G}_{n-1}$ by Theorem~\ref{THMpath}.  Also, since $\mathcal{S}_n$ is in bijection with the set of plane-oriented recursive paths of size $n+1$ (see Janson~\cite{Jan2008}), which in turn is in bijection with the set of path diagrams $\mathcal{F}_n$ by Theorem~\ref{THMpath2}, we get the identity $1 + S(t)= f(t) = 1+t g(t)$.

The first of the branched continued fractions is known as a \L ukasiewicz fraction; its form is obtained through the characteristic polynomial of \L ukasiewicz paths in a manner similar to that in section~\ref{SEC ogf local types}, taking into account that there is no restriction on the type of ``rise" steps (for more on \L ukasiewicz paths, see the works of Roblet~\cite{Roblet1994} and Viennot~\cite{Viennot1983}, and also~\cite{Varvak2008}).
\end{proof}

\subsection{Equivalent statistics on plane-oriented recursive trees, ternary trees, and Stirling permutations}

Let $X_{n,j}$ denote the number of nodes of outdegree $j$ in a random plane-oriented recursive tree of size $n$.
We relate the distribution of outdegrees to suitably defined statistics in ternary increasing trees
and Stirling permutations. Any ternary increasing tree can be decomposed by deleting all center edges into trees having only left or right edges.
Let $X^{[LR]}_{n,j}$ denote the number of size $j$ left-right trees in a random ternary increasing tree of size $n$.

\smallskip

Concerning Stirling permutations $\sigma=\sigma_1\dots\sigma_{2n}$, we introduce block structures as follows.
A block in a Stirling permutation $\sigma$ is a substring
$\sigma_p\dotsm \sigma_q$ with $\sigma_p=\sigma_q$ that is maximal, i.e.~not contained in any larger
such substring~\cite{JanKuPa2008}. There is at most one block for every $i=1,\dots,n$,
extending from the first occurrence of $i$ to the last; we say that
$i$ forms a block when this substring
is not contained in a string $\ell\dotsm \ell$ for some $\ell<i$. The decomposition $\sigma=[B_1][B_2]\dots[B_j]$ is a block structure of $\sigma$. Removing from each of the blocks the leftmost and the rightmost number, we are left with possibly empty substrings, which after an order-preserving relabeling form (sub-)Stirling permutations. We recursively determine the block structure in these (sub)-Stirling permutations.
The Stirling permutation $\sigma$ has a block structure of size $j$ if either $\sigma$ or any of the recursively obtained (sub)-Stirling permutations decompose into $j$ blocks.
Let $X^{[B]}_{n,j}$ denote the number of block structures of size $j$ in a random Stirling permutation of size $n$.

\begin{example}
The Stirling permutation $\sigma=221553367788614499$ of size nine
has block decomposition $\sigma=[22][155336778861][44][99]$ of size 4. After
removal of the leftmost and the rightmost entries in the blocks, the only
non-empty (sub-)Stirling permutation is given by
$5533677886$. After an order-preserving relabeling we get
$\sigma'=2211344553$. We have $\sigma'=[22][11][344553]$, a block decomposition of size 3;
consequently we obtain the (sub-)Stirling permutation
$\sigma''=1122$, which has block decomposition $\sigma''=[11][22]$ of size 2.
Hence, $X^{[B]}_{9,4}(\sigma)=1$, $X^{[B]}_{9,3}(\sigma)=1$ and
$X^{[B]}_{9,2}(\sigma)=1$.
\end{example}

\begin{theorem}
For $j>2$, the distribution of the number of nodes $X_{n+1,j}$ of outdegree $j$ in a random plane-oriented recursive tree of size $n+1$ coincides with the distribution of  the number $X^{[LR]}_{n,j-1}$ of size $j-1$ left-right trees in a random ternary increasing tree of size $n$, and
with the distribution of the number $X^{[B]}_{n,j}$ of sub-Stirling permutations with number of blocks equal to $j$ in a random Stirling permutation of size $n$.


Moreover, the nodes of outdegree two in plane-oriented recursive tree of size $n+1$ correspond to
the number of nodes in ternary increasing trees of size $n$ which have exactly one child connected by a center edge,
such that this child is itself a leaf node.
\end{theorem}

The proof of the the result consists of a simple application of the bijection stated in~\cite{JanKuPa2008}, and is therefore omitted.

\begin{figure}[!htb]
\centering
\includegraphics[angle=0,scale=0.9]{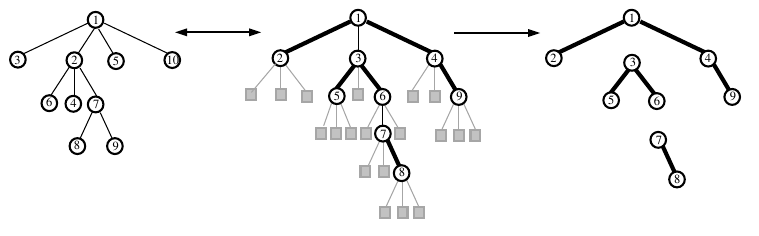}
\caption{A plane-oriented increasing tree of size 10, the corresponding size 9 ternary increasing trees together with its left-right tree decomposition.\label{fig2}}
\end{figure}

\begin{example}
The Stirling permutation $\sigma$ of size nine corresponding to the trees in \refF{fig2}, obtained either using the bijection
with plane-oriented recursive tree \cite{Jan2008} or with ternary increasing trees \cite{JanKuPa2008},
is given by $\sigma=221553367788614499$. As observed before we have $X^{[B]}_{9,4}(\sigma)=1$, $X^{[B]}_{9,3}(\sigma)=1$ and $X^{[B]}_{9,2}(\sigma)=1$, corresponding to the number of nodes with outdegrees given by four, three and two in the corresponding plane-oriented recursive trees, and
with the sizes of the left-right trees in ternary increasing trees.
\end{example}

\subsection{General case of $k$-Striling permutations and other tree families} 
In~\cite{JanKuPa2008}, Janson et al. presented a different tree family in bijection with $(k+1)$-Stirling permutations, and presented a more general case of $k$-Stirling permutations. We summarize the main correspondences and re-express some of the results by using systems of labeled \L ukasiewicz path diagrams. Refer to~\cite{JanKuPa2008} and references therein for more detailed definitions of the different tree families.

\medskip

The increasing tree family $\cB_n(k)$ of $k$-bundled increasing trees of order $n$ are increasing plane trees where each node has an additional $k-1$ separation walls, which can be regarded as a type of special edges (half-edges) without child nodes. The case $k=1$ corresponds to the ordinary plane-oriented recursive trees. 

This family's degree-weight generating function
$\varphi(t)=\frac{1}{(1-t)^{k}}$. Consequently, by solving~\eqref{eqnz1} one gets the generating function $T(z)$ and the total weight $T_n$:
\begin{equation*}
T(z)=1-(1-(k+1)z)^{\frac1{k+1}}-1 ,\qquad T_{n}= \prod_{l=1}^{n-1}(l(k+1)-1).
\end{equation*}

The result of~\cite{JanKuPa2008} says that
\[
\Seq(\cB_n(k))\cong \mathcal{Q}_n(k+1).
\]
By definition, the exponential generating function of $\Seq(\cB_n(k))$s
is given by 
\[
F(z)=\frac{1}{1-T(z)}=\frac{1}{(1-(k+1)z)^{\frac1{k+1}}}
\]
and coincides with the generating function of $(k+2)$-ary increasing trees (except for the initial value $F(0)=1$), see~\eqref{eqnkary} and thus with $(k+1)$-Stirling permutations.

\begin{theorem}[\cite{JanKuPa2008}]
The family of $\Seq(\cB_n(k))$ of sequences of $k$-bundled increasing trees is in bijection with $(k+2)$-ary increasing trees and thus with $(k+1)$-Stirling permutations. 
\end{theorem}

\smallskip

In the following we will realize $\Seq(\cB_n(k))$ as a family $\mathcal{U}=\mathcal{U}(k)$ of non-standard increasing trees,  where the root node has a different degree-weight generating function compared to the rest of nodes:
The family $\mathcal{U}_{k}$ is specified by
\[
\mathcal{U}= \bigcirc\hspace*{-0.75em}\text{\small{$1$}}\hspace*{0.3em} \times \vartheta(\mathcal{T}),
\quad \mathcal{T}= \bigcirc\hspace*{-0.75em}\text{\small{$1$}}\hspace*{0.3em} \times \varphi(\mathcal{T}),
\]
with $\varphi(t)=\frac{1}{(1-t)^k}$ and $\vartheta(t)=\frac1{1-t}$. 
Hence, we obtain
\[
U'(z)=F(z)=\frac{1}{(1-(k+1)z)^{\frac1{k+1}}}.
\]
Let $U_n=|\mathcal{U}_{n}(k)|$. Thus, $U_{n+1}=Q_n$ and the bijection between $\Seq(\cB_n(k))$ and $\mathcal{Q}_n(k+1)$
translates into a bijection $\mathcal{U}_{n+1}(k)\cong \mathcal{Q}_n(k+1)$, generalizing the bijection between ordinary plane-oriented recursive trees and Stirling permutations~\cite{Jan2008,JanKuPa2008} (case $k=1$).

\begin{theorem}
The family of $\Seq(\cB(k))$ of sequences of $k$-bundled increasing trees 
can be realized as non-standard increasing trees $\mathcal{U}(k)$ with two degree-weight generating functions:
$\varphi(t)=\frac{1}{(1-t)^k}$ for non-root nodes and $\vartheta(t)=\frac1{1-t}$ for the root. 
\end{theorem}

A straightforward extension of our previous description of plane-oriented recursive trees, case $k=1$, gives the following result.

\begin{theorem}
\label{THMpath3}
The class of $\mathcal{U}_{n+1}(k)$ of increasing trees is in bijection with the system of labeled \L ukasiewicz path diagrams whose plane paths are comprised of $n$ steps, with possibility function $\pos(.)$ given as follows: 

\begin{equation*}
\begin{split}
\pos(a^\ell_0)&=1, \\
\pos(a^\ell_j)&=\binom{\ell+k}{k-1}\cdot(j+1) \quad \text{ for } \quad j>0, \\
\pos(b_0)&=j+1.
\end{split}    
\end{equation*}

\end{theorem}

\smallskip 

Another tree family $\mathcal{V}=\mathcal{V}(k)$ in bijection with $(k+1)$-Stirling permutations can be obtained as follows: we combine a root node, weighted according to a $(k+2)$-bundled increasing tree, with subtrees weighted according to $k$-bundled increasing trees. 
\[
\mathcal{V}= \bigcirc\hspace*{-0.75em}\text{\small{$1$}}\hspace*{0.3em} \times \vartheta(\mathcal{T}),
\quad \mathcal{T}= \bigcirc\hspace*{-0.75em}\text{\small{$1$}}\hspace*{0.3em} \times \varphi(\mathcal{T}),
\]
with $\varphi(t)=\frac{1}{(1-t)^k}$ and $\vartheta(t)=\frac1{(1-t)^{k+2}}$. 
Thus, 
\[
V'(z)=\vartheta(T(z))=\frac{1}{(1-(k+1)z)^{1+\frac1{k+1}}},\quad 
V(z)=\frac{1}{(1-(k+1)z)^{\frac1{k+1}}}-1.
\]
Let $\mathcal{V}_{n}(k)$ denote the family of trees of $\mathcal{V}(k)$ with $n$ nodes. By extraction of coefficients we observe that  $V_n(k)=|\mathcal{V}_{n}(k)|=Q_n(k+1)$. We note that a bijection between $\mathcal{V}_n(k)$ and $\mathcal{Q}_n(k+1)$ can be obtained similarly to Janson's original bijection~\cite{Jan2008}, proceeding via a depth-first walk, also taking into account the separation walls of the bundled trees. See also the work of Janson et al.~\cite{JanKuPa2008} for a closely related bijection between $k$-bundled increasing trees and so-called $k$-bundled Stirling permutations. We summarize our findings.

\begin{theorem}
The family of $\mathcal{V}(k)$ of non-standard increasing trees can be constructed with two degree-weight generating functions:
$\varphi(t)=\frac{1}{(1-t)^{k+2}}$ for non-root nodes and $\vartheta(t)=\frac1{(1-t)^k}$ for the root. $\mathcal{V}(k)$ is in bijection with $(k+1)$-Stirling permutations and also with $(k+2)$-ary increasing trees.
\end{theorem}

\begin{proof}
We label each separation wall of a node labeled $v$ by the label of the node $v$. 
Additionally, we label any ordinary edge by the label of the child. 
Any non-root node has at least $k-1$ outgoing edges, thinking
of the separation walls as a special type of edges. Moreover, it has one incoming edge from its ancestor.   
The root has by definition $k+1$ separation walls. Now we perform the depth-first walk and code the tree by the sequence of the labels visited on the edges, under the additional rule that a label on a separation wall only contributes once. 
Since every proper edge is traversed twice, and every label except 1
occurs on exactly one proper edge, every integer appears exactly $k+1$ times and the code is 
a permutation of the multiset $\{1^{k+1}, 2^{k+2},\dots, n^{k+2}\}$. 
By construction, the elements occurring between the two
occurrences of $i$ are larger than $i$, since we can only visit
nodes with higher labels.
\end{proof}

\begin{theorem}
\label{THMpath4}
The class of $\mathcal{V}_{n+1}(k)$ of increasing trees is in bijection with the system of labeled \L ukasiewicz path diagrams whose plane paths are comprised of $n$ steps, with possibility function $\pos(.)$ given as follows:

\begin{equation*}
\begin{split}
\pos(a^\ell_0)&=\binom{\ell+k+2}{k+1}, \\
\pos(a^\ell_j)&=\binom{\ell+k}{k-1}\cdot(j+1) \quad \text{ for } \quad j>0, \\
\pos(b_j)&=j+1.
\end{split}
\end{equation*}

\end{theorem}

\begin{proof}
As the construction is similar to our constructions given before, we will be more brief. Given a path diagram $(L,p)$, we begin with one placeholder for an internal node of the tree $T$. Concerning the root and the first step of the path, it has to be a ``rise" step $L_i=a^\ell_0$: we enter the label one at the root and distribute $\ell$ placeholders into the $k+2$ different positions (induced by the $k+2$ separation walls stemming from the root) according to the value of $p_1$ and a lexicographical ordering. 

\smallskip

For the remaining steps $2\le i \le n$ we proceed as follows: 

\begin{itemize}
    \item If the $i$-th step of the path is a ``fall" step $L_i=b_j$: using depth-first order, choose the $(p_i+1)$-st available placeholder. Replace it with a node labeled by $i$, itself having no child nodes.
    \item If the $i$-th step of the path is a ``rise" step $L_i=a^\ell_j$: Let $s$ be the number of integer times $\binom{\ell+k}{k-1}$ goes into $p_i$, and let $t$ be the remainder of $p_i$ modulo $\binom{\ell+k}{k-1}$, such that $p_i=s\cdot \binom{\ell+k}{k-1}+t$. Using depth-first order, choose the $(s+1)$-st available placeholder and enter a node labelled $i$ with $k$ 
    separation walls. Distribute $\ell$ placeholders into the $k$ different positions (induced by the $k$ separation walls) according to the value of $t$ and a lexicographical ordering. 
\end{itemize}

After $n$ steps, there will be a single available placeholder remaining.  To complete the tree, replace that placeholder with a node labeled by $n+1$ (with no child nodes).

\end{proof}

A direct consequence of Theorems~\ref{THMpath}, \ref{THMpath3} and~\ref{THMpath4} is the following observation. 

\begin{coroll}
There exist at least three different branched continued fractions expansions of the generating function
of $k$-Stirling permutations $\sum_{k\ge 0}Q_n(k)t^n= \sum_{n\ge 0}k^{n}\frac{\Gamma(n+1+\frac1k)}{\Gamma(\frac1k)}t^n$ due to three different tree families and their path diagram representations.
\end{coroll}

\begin{remark}
We finish by giving a symbolic overview of the involved bijections between $k$-Stirling permutations and increasing tree families.
\begin{equation*}
 \mathcal{Q}_n(k) \cong 
 \begin{cases} 
 \mathcal{T}_n(k+1),\\
 \mathcal{U}_{n+1}(k-1)\cong \Seq(\mathcal{B})_n(k-1),\\
 \mathcal{V}_n(k-1).
 \end{cases}
\end{equation*}
Our contributions in this work are the bijections to the family $\mathcal{U}$ and $\mathcal{V}$.
Note that for $k=1$ the second and third families are not defined. Moreover, for $k=2$ the families
$\mathcal{U}_{n+1}(1)\cong \Seq(\mathcal{B})_n(1)\cong \mathcal{P}_{n+1}$ are simply ordinary plane-oriented recursive trees.
\end{remark}

\end{document}